\documentclass[12pt]{amsart}%
\usepackage{bbding}
\usepackage{dsfont}
\usepackage{graphicx}
\usepackage{amsthm,amscd}
\usepackage{calc}
\usepackage{hyperref}
\usepackage{mathrsfs}
\usepackage{mathrsfs}
\usepackage{CJK,fancyhdr,amscd}
\usepackage{anysize}
\usepackage{amsmath,amssymb,amsfonts}
\usepackage{epsfig}
\usepackage{latexsym}
\usepackage{hyperref}
\usepackage{indentfirst, latexsym}
\usepackage{graphics}
\usepackage{amsmath}
\usepackage{amsfonts}
\usepackage{amssymb}%
\providecommand{\U}[1]{\protect\rule{.1in}{.1in}}

\numberwithin{equation}{section}
\providecommand{\U}[1]{\protect\rule{.1in}{.1in}}
\newtheorem{theorem} {Theorem} [section]
\newtheorem{proposition}[theorem]{Proposition}
\newtheorem{corollary}  [theorem]     {Corollary}
\newtheorem{lemma}  [theorem]     {Lemma}

\newtheorem{definition}  [theorem]     {Definition}

\newcommand{\qtq}[1]{\quad\mbox{#1}\quad}

\newcommand{\G}{\mathbb{G}}
\newcommand{\db}{\overline{\partial}}
\newcommand{\dbs}{\overline{\partial}^*}

\newcommand{\lc}{\lrcorner}

\newcommand{\sqa}{\square}
\newcommand{\osqa}{\overline{\square}}

\newcommand{\Om}{\Omega}
\newcommand{\btheorem}{\begin{theorem}}
\newcommand{\etheorem}{\end{theorem}}
\newcommand{\bproposition}{\begin{proposition}}
\newcommand{\eproposition}{\end{proposition}}
\newcommand{\bdefinition}{\begin{definition}}
\newcommand{\edefinition}{\end{definition}}
\newcommand{\bcorollary}{\begin{corollary}}
\newcommand{\ecorollary}{\end{corollary}}
\newcommand{\bproof}{\begin{proof}}
\newcommand{\eproof}{\end{proof}}
\newcommand{\beq}{\begin{equation}}
\newcommand{\eeq}{\end{equation}}
\newcommand{\ee}{\end{eqnarray*}}
\newcommand{\be}{\begin{eqnarray*}}

\newcommand{\elemma}{\end{lemma}}
\newcommand{\blemma}{\begin{lemma}}
\newcommand{\ts}{\otimes}
\newcommand{\om}{\omega}
\renewcommand{\>}{\rightarrow}
\newcommand{\sL}{{\mathcal L}}
\newcommand{\p}{\partial}
\newcommand{\bp}{\overline\partial}
\renewcommand{\H}{\mathbb H}
\newcommand{\la}{\langle}
\newcommand{\ra}{\rangle}

\newcommand{\bd}{\begin{enumerate} }
\newcommand{\ed}{\end{enumerate}}

\begin{document}
\title{Quasi-isometry and deformations of Calabi-Yau manifolds}
\author{Kefeng Liu}
\address{Center of Mathematical Sciences, Zhejiang University, Hangzhou, China;
Department of Mathematics, University of California at Los Angeles,
Angeles, CA 90095-1555, USA}
\email{liu@math.ucla.edu,liu@cms.zju.edu.cn}
\author{Sheng Rao}
\address{Center of Mathematical Sciences, Zhejiang University, Hangzhou 310027, China}
\email{likeanyone@zju.edu.cn}

\author{Xiaokui Yang}
\address{Department of Mathematics, Northwestern University,
Evanston, IL 60208, USA} \email{xkyang@math.northwestern.edu}
\dedicatory{In memory of Professor Andrey Todorov}

\date{\today}

\subjclass[2010]{Primary 32G05; Secondary 58A14, 53C55, 14J32}
\keywords{Deformations of complex structures, Hodge theory,
Hermitian and K\"ahlerian manifolds, Calabi-Yau manifolds}

\begin{abstract}
We prove several formulas related to Hodge theory and the
Kodaira-Spencer-Kuranishi deformation theory of K\"ahler manifolds.
As applications, we present a construction of globally convergent
power series of integrable Beltrami differentials on Calabi-Yau
manifolds and  also a construction of global canonical family of
holomorphic $(n,0)$-forms on the deformation spaces of Calabi-Yau
manifolds. Similar constructions are also  applied to the
deformation spaces of compact K\"ahler manifolds.
\end{abstract}
\maketitle

\section{Introduction}

In this paper, we will present several results about Hodge theory
and the deformation theory of Kodaira-Spencer-Kuranishi on compact
K\"{a}hler manifolds. Our main observations include a simple
$L^{2}$-quasi-isometry result for bundle valued differential forms,
an explicit formula for the deformed $\bar\partial$-operator, and an
iteration method to construct global Beltrami differentials on
Calabi-Yau (CY) manifolds and holomorphic $(n,0)$-forms on the
deformation spaces of compact K\"{a}hler manifolds of dimension $n$.
We will present an alternative simple method to solve the
$\overline\partial$-equation,
 prove global convergence of the formal power
series of the Beltrami differentials and the holomorphic
$(n,0)$-forms constructed from the Kodaira-Spencer-Kuranishi
theory. These series previously were only proved to converge
in an arbitrarily small neighborhood. We will
discuss more applications to the Torelli problem and the
extension of twisted pluricanonical sections in a sequel to this paper.

Let us first fix some notations to be used throughout this paper.
All  manifolds in this paper are assumed to be compact and K\"ahler,
though some results still hold for complete K\"ahler manifolds; a
Calabi-Yau, or {CY manifold}, is a compact projective manifold with
trivial canonical line bundle. By Yau's solution to the Calabi
conjecture, there is a CY metric on $X$ such that the holomorphic
$(n,0)$-form $\Omega_{0}$ on $X$ is parallel with respect to the
metric connection. For a complex manifold $(X,\omega)$ and a
Hermitian holomorphic vector bundle $(E,h)$ on $X$, we denote by
$A^{p,q}(X)$ the space of smooth $(p, q)$-forms on $X$ and by
$A^{p,q}(E)= A^{p,q}(X, E)$ the space of smooth $(p, q)$-forms on
$X$ with values in $E$. Similarly, let $\mathbb{H}^{p,q}(X)$ be the
space of the harmonic $(p, q)$-forms and let $\mathbb{H}^{p,q}(X,E)$
be the space of the harmonic $(p, q)$-forms with values in $E$. Let
$\nabla$ be the Chern connection on $(E,h)$ with canonical
decomposition $\nabla=\nabla^{'}+\bp$ where $\nabla'$ is the $(1,0)$
part of the Chern connection $\nabla$. Let $\mathbb{G}$ and
$\mathbb{H}$ denote the Green operator and harmonic projection in
the Hodge decomposition  with respect to the operator $\bp$, that is
$$\mathbb I=\H+(\bp\bp^*+\bp^*\bp)\G$$ A {Beltrami differential} is
an element in $A^{0,1}(X, T^{1,0}_X)$, where $T^{1,0}_X$ denotes the
holomorphic tangent bundle of $X$.
 The $L^{2}$-norm $\Vert\cdot\Vert=\Vert\cdot
\Vert^{\frac{1}{2}}_{L^{2}}$ is induced by the  metrics $\omega$ and
$h$.  The $\mathscr{C}^{k}$-norm $\Vert\cdot\Vert_{\mathscr{C}^k}$
will be used on the Beltrami differentials.

Now we briefly describe the main results in this paper.  The
following quasi-isometry on compact K\"ahler manifolds is obtained
in Section \ref{dbeqn}.

\begin{theorem}[Quasi-isometry]\label{1main} Let $(E,h)$ be a Hermitian holomorphic vector bundle  over the compact K\"ahler manifold $(X,\omega)$. \bd \item
For any $g\in A^{n,\bullet}(X,E)$, we have the following estimate
$$\|\bp ^*\mathbb{G}g\|^2\leq \la  g, \G g\ra.$$

\item If $(E,h)$ is a strictly positive line bundle with Chern curvature $\Theta^E$ and $\omega=\sqrt{-1}\Theta^E$, for any $g\in A^{n-1,\bullet}(X,E)$ we obtain
$$ \|\bp^*\G\nabla' g\|\leq \|g\|.$$

\item If $E$ is the trivial line bundle,  for any smooth  $g\in A^{p,q}(X)$, $$\|\bp^*\G\p
g\|\leq \|g\|.$$ In particular, if $\bp\p g=0$ and $g$ is
$\p^*$-exact, we obtain the isometry
$$\|\bp^*\G\p g\|=\|g\|.$$

\ed\end{theorem}

\noindent Here the operator $\bp^*\G$ can be viewed as the ``inverse
operator" of $\bp$. More precisely, we can write down the explicit
solutions of some $\bp$-equations by using $\bp^*\G$, which can also be
considered as a bundle-valued version of the very useful $\p\bp$-lemma in complex geometry.

\bproposition\label{1eq} Let $(E,h)$ be a Hermitian holomorphic
vector bundle with semi-Nakano positive curvature tensor $\Theta^E$
over the compact K\"ahler manifold $(X,\omega)$. Then, for any $g\in
A^{n-1,\bullet}(X,E)$ with $\bp  \nabla'g =0$, the $\bp$-equation
${\bp } s = \nabla'g$ admits a solution  $$s = {\bp }^* \mathbb{G}
\nabla'g,$$ such that $$\|s\|^2\leq \la \nabla' g, \G  \nabla'
g\ra.$$
  Moreover, this
solution is unique if we require $\mathbb{H}(s)=0$ and
$\overline{\partial}^{\ast}s=0$.

\eproposition

Note that, in the proofs of Theorem \ref{1main} and Proposition
\ref{1eq}, we only use basic  Hodge theory, so they still hold on
general K\"ahler manifolds as long as Hodge theory can be applied.
On the other hand,  in Proposition \ref{1eq}, the curvature
$\Theta^E$ is only required to be semi-positive and it is
significantly different from all variants of H\"ormander's
$\bp$-estimates. Moreover, Proposition \ref{1eq} can also hold if
$h$ is a singular Hermitian metric, and the curvature $\Theta^E$ has
certain weak positivity in the current sense.

\vskip 1\baselineskip

In the following,  we shall use $i_\phi$ and  $\phi\lrcorner$ to
denote the contraction operator with $\phi\in A^{0,1}(X,T^{1,0}_X)$
alternatively if there is no confusion. For $\phi\in
A^{0,1}(X,T^{1,0}_X)$,  the Lie derivative can be lifted to act on
bundle valued forms by
$$\sL_\phi=-\nabla \circ i_\phi+i_\phi\circ \nabla.$$
There is also a canonical decomposition
$$\sL_\phi=\sL^{1,0}_\phi+\sL_\phi^{0,1}$$
according to the types.

In Section \ref{Bel}, we prove some explicit formulas for the
deformed differential operators on the deformation spaces of complex
structures and one of our main results is \btheorem\label{1main2}
Let $\phi\in A^{0,1}(X,T^{1,0}_X)$. Then on the space
$A^{\bullet,\bullet}(X,E)$, we have $$ e^{- i_{\phi}}\circ
\nabla\circ e^{ i_\phi}=\nabla -\sL_{\phi}-\
i_{\frac{1}{2}[\phi,\phi]}=\nabla-\sL_\phi^{1,0}+i_{\bp\phi-\frac{1}{2}[\phi,\phi]}.$$
In particular, if $\sigma\in A^{n,\bullet}(X,E)$ and $\phi$ is
integrable, i.e., $\bp\phi-\frac{1}{2}[\phi,\phi]=0$, then
$$ \left(e^{- i_\phi}\circ \nabla \circ
e^{ i_\phi}\right)(\sigma)=\bp\sigma+\nabla'(\phi\lrcorner
\sigma).$$
\etheorem

As applications of Theorem \ref{1main} and Theorem \ref{1main2}, in
Section \ref{gcfB} we use ideas of recursive methods to construct
Beltrami differentials in Kodaira-Spencer-Kuranishi deformation
theory. Similar methods are also presented in \cite{A}, \cite{[MK]},
\cite{G}, \cite{[To89]},   \cite{[T]},\cite{C}, \cite{Siu86},
\cite{Schumacher85} and the references therein.
  At
first, we  present the following global convergence
 on the deformation space of CY manifolds:

\begin{theorem}
\label{0Phi} Let $X$ be a CY manifold and $\varphi_{1}\in\mathbb{H}%
^{0,1}(X,T^{1,0}_X)$ with norm
$\|\varphi_{1}\|_{\mathscr{C}^1}=\frac{1}{4C_1}$. Then for any
nontrivial holomorphic $(n,0)$ form $\Omega_0$ on $X$, there exits a
smooth globally convergent power series for $|t|<1$,
$$\Phi(t)=\varphi_{1}t^{1}+\varphi_{2}t^{2}+\cdots+\varphi_{k}%
t^{k}+\cdots\in A^{0,1}(X,T^{1,0}_X),
$$
which satisfies:
\bd
\item $\overline{\partial}\Phi(t)=\frac{1}{2}[\Phi(t),\Phi(t)]$;

\item $\overline{\partial}^{*}\varphi_{k}=0$\ {for each $k\geq1$};

\item $\varphi_{k}\lrcorner\Omega_{0}$ is $\partial$-exact\ for each
$k\geq2$;

\item $\|\Phi(t)\lrcorner\Omega_{0}\|_{L^{2}}<\infty$ as long as
$|t|<1$. \ed
\end{theorem}

\noindent The key ingredient in Theorem \ref{0Phi} is that the
convergent radius of the power series is at least $1$, which was
previously proved to be sufficiently small. We shall see that the
$L^2$-estimate in Theorem \ref{1main} plays a key role in the proof
of Theorem \ref{0Phi}. The power series thus obtained is called an
{$L^{2}$-global canonical family of Beltrami differentials} on the
CY manifold $X$.

\vskip 1\baselineskip

In Section \ref{gcfh}, we obtain the following theorem to construct
deformations of holomorphic $(n, 0)$-forms, which are globally
convergent in  the $L^2$-norm for CY manifolds.
\begin{theorem}
\label{gcfn0}Let $\Omega_0$ be a nontrivial holomorphic $(n,0)$-form
on the CY manifold $X$ and  $X_{t}=(X_{t}, J_{\Phi(t)})$ be the
deformation of the CY manifold $X$ induced by $\Phi(t)$ as
constructed in Theorem \ref{0Phi}. Then for any $|t|< 1$, $$\Omega
_{t}^{C}:=e^{\Phi(t)}\lrcorner\Omega_{0}$$ defines an $L^{2}$-global
canonical family of holomorphic $(n,0)$-forms on $X_{t}$.
\end{theorem}

\noindent  As a straightforward consequence of Theorem \ref{gcfn0},
we have the following global expansion of the canonical family of
$(n,0)$-forms on the deformation spaces of CY manifolds in
cohomology classes. Similar ideas are   also used in \cite[~Theorem
~1.34]{G}. This expansion also has interesting applications in
studying the global Torelli problem.

\begin{corollary}
\label{edeR} With the same notations as in Theorem \ref{gcfn0}, there holds
the following global expansion of $[\Omega_{t}^{C}]$ in cohomology classes for
$|t|<1$
\begin{equation}
[\Omega_{t}^{C}]=[\Omega_{0}]+\sum_{i=1}^{N}[\varphi_{i}\lrcorner\Omega
_{0}]t_{i}+O(|t|^{2}),
\end{equation}
where $O(|t|^{2})\in\displaystyle\bigoplus_{j=2}^{n} H^{n-j,j}(X)$ denotes the
terms of orders at least $2$ in $t$.
\end{corollary}

 Finally, we need to point out that on the deformation spaces of
 compact K\"ahler manifolds, if we assume the existence of a global
 family of Beltrami differentials
$\Phi(t)$ as stated in Theorem \ref{0Phi}, we can also construct
{$L^{2}$-global  family of $(n,0)$-forms on the deformation spaces
of compact K\"ahler manifolds. For more details, see Theorem
\ref{m1} and Corollary \ref{edeRK1}.

\vspace{0.1cm} $\mathbf{Acknowledgement}$ This paper originated from
many discussions with Prof. Andrey Todorov, who unexpectedly passed
away in March 2012 during his visit of Jerusalem. We dedicate this
paper to his memory. The second author would also like to express
his gratitude to Weijun Lu, Quanting Zhao and Shengmao Zhu for their
interest and useful comments.

\section{$\bp$-equations on non-negative vector bundles}\label{dbeqn}
 In this section, we will prove a quasi-isometry result in $L^{2}%
$-norm with respect to the operator ${\overline{\partial}} ^{\ast}%
\circ\mathbb{G}$ on a compact K\"ahler manifold. This gives a rather
simple and explicit way to solve vector bundle valued $\bp
$-equations with $L^2$-estimates.

Let $(E,h)$ be a Hermitian holomorphic vector bundle over the compact
K\"ahler manifold $(X,\omega)$ and $\nabla= \nabla' +\bp$ be
the Chern connection on it. With respect to metrics on $E$ and $X$,  we set
$$\osqa=\db\dbs+\dbs\db,$$
$$\sqa'=\nabla'\nabla'^*+\nabla'^*\nabla'.$$
Accordingly, we associate the Green operators and harmonic
projections $\mathbb{G}$, $\mathbb{H}$ and $\mathbb{G}'$,
$\mathbb{H}'$ in Hodge decomposition to them, respectively. More
precisely,
$$\mathbb I=\mathbb{H}+\overline\square\circ \mathbb G,~~~~~~~~~~~~\mathbb I=\mathbb{H}'+\square'\circ \mathbb G'.$$

Let $\{z^i\}_{i=1}^n$ be  the local holomorphic coordinates
  on $X$ and  $\{e_\alpha\}_{\alpha=1}^r$ be a local frame
 of $E$. The curvature tensor $\Theta^E\in \Gamma(X,\Lambda^2T^*X\ts E^*\ts E)$ has the form
 \beq \Theta^E= R_{i\bar j\alpha}^\gamma dz^i\wedge d\bar z^j\ts e^\alpha\ts e_\gamma,\eeq
where $R_{i\bar j\alpha}^\gamma=h^{\gamma\bar\beta}R_{i\bar
j\alpha\bar \beta}$ and \beq R_{i\bar j\alpha\bar\beta}= -\frac{\p^2
h_{\alpha\bar \beta}}{\p z^i\p\bar z^j}+h^{\gamma\bar
\delta}\frac{\p h_{\alpha \bar \delta}}{\p z^i}\frac{\p
h_{\gamma\bar\beta}}{\p \bar z^j}.\eeq Here and henceforth we
 adopt the Einstein convention for summation.

\bdefinition
 A Hermitian vector bundle
$(E,h)$ is said to be \emph{semi-Nakano-positive} (resp.
\emph{Nakano-positive}), if for any nozero vector
$u=u^{i\alpha}\frac{\p}{\p z^i}\ts e_\alpha$, \beq
\sum_{i,j,\alpha,\beta}R_{i\bar j\alpha\bar \beta} u^{i\alpha}\bar
u^{j\beta}\geq 0,~~~ (resp. >0). \eeq For a line bundle, it is
strictly positive if and only if it is Nakano-positive. \edefinition

\begin{theorem}[Quasi-isometry] Let $(E,h)$ be a Hermitian holomorphic vector bundle  over the compact K\"ahler manifold $(X,\omega)$. \bd \item
For any $g\in A^{n,\bullet}(X,E)$, we have the following estimate
$$\|\bp ^*\mathbb{G}g\|^2\leq \la  g, \G g\ra.$$

\item If $(E,h)$ is a strictly positive line bundle and $\omega=\sqrt{-1}\Theta^E$, for any $g\in
A^{n-1,\bullet}(X,E)$,
$$ \|\bp^*\G\nabla' g\|\leq \|g\|$$

\item If $E$ is the trivial line bundle,  for any smooth  $g\in A^{p,q}(X)$, $$\|\bp^*\G\p
g\|^2=\|g\|^2- \|\H(g)\|^2-\left\la \p^* g, \G(\p^*
g)\right\ra-\|\G(\bp\p g)\|^2\leq \|g\|^2.$$ In particular, if
$\bp\p g=0$ and $g$ is $\p^*$-exact, we obtain the isometry
$$\|\bp^*\G\p g\|=\|g\|.$$

\ed\end{theorem} \bproof  (1). For $g\in A^{n,\bullet}(X,E)$, \be
   \|\bp ^*\mathbb{G}g\|^2&
 =&\langle \bp \bp ^*\mathbb{G}g, \mathbb{G}g\rangle \\
 &=&\langle g, \mathbb{G}g \rangle -\langle \bp ^*\bp \mathbb{G}g , \mathbb{G}g\rangle -\langle \mathbb{H}g,
 \mathbb{G}g\rangle \\
&=&\langle g, \mathbb{G} \rangle -\langle
\bp \mathbb{G}g ,
\bp \mathbb{G}g\rangle  \\
&\leq& \langle g, \mathbb{G} g \rangle\ee since the Green operator
is self-adjoint and zero on the kernel of Laplacian by definition.

(2). If $(E,h)$ is a strictly positive line bundle over $X$ and
$\omega=\sqrt{-1}\Theta^E$,  for any $g\in A^{n-1,q}(X,E)$, by the
well-known Bochner-Kodaira-Nakano identity $ \osqa= \sqa' +
[\sqrt{-1}\Theta^E, \Lambda_\omega]$,
$$\overline\square (\nabla'g)=\square'(\nabla' g)+q (\nabla' g)=(\square'+q) (\nabla' g),$$
we obtain $\H(\nabla' g)=0$ and thus $\overline\square \G(\nabla'
g)=\nabla' g=\square' \G'(\nabla' g)$ since obviously $\H'(\nabla'
g)=0$ by Hodge decomposition. Moreover, \be \la \nabla' g,
\G(\nabla' g)\ra &=&\la\nabla' g, \overline \square^{-1} (\nabla'
g)\ra\\&=&\la\nabla' g, (\square'+q)^{-1}(\nabla' g)\ra\\&\leq& \la
\nabla' g, \square'^{-1} (\nabla' g)\ra\\&=& \la \nabla' g,
\G'(\nabla' g)\ra.\ee Therefore, \be \|\bp^*\G\nabla' g\|^2&\leq&\la
\nabla' g, \G\nabla' g\ra\\
&\leq& \la\nabla' g, \G'\nabla'g\ra\\&=&\la g, \nabla'^*\nabla' \G' g\ra\\
&=&\la g, g-\H'(g)- \nabla'\nabla'^* \G' g\ra\\
&=&\|g\|^2-\|\H'(g)\|^2-\la \nabla'^* g, \G'\nabla'^*g \ra\\
&\leq & \|g\|^2. \ee

 (3). If $E$ is the trivial line bundle, for any $g\in
A^{p,q}(X)$,
 we have the
following \be \|\bp^*\G\p g\|^2&=&\left\la \bp^*\G\p g, \bp^*\G\p
g\right\ra=\left\la\bp\bp^*\G\p g, \G\p g \right\ra \\
&=&\left\la \overline\square\G\p g-\bp^*\bp \G\p g, \G\p g \right\ra\\
&=&\left\la \p g, \G\p g\right\ra-\left\la \bp^*\bp \G\p g, \G\p g
\right\ra\\
&=&\left\la g, \p^*\p \G g\right\ra-\left\la \G\bp\p g, \G\bp\p g
\right\ra\\
&=&\left\la g, \square' \G g- \p\p^*\G g\right\ra-\|\G(\bp\p
g)\|^2\\
&=&\left\la g, g-\H(g)-\p\p^* \G g\right\ra-\|\G(\bp\p g)\|^2\\
&=&\|g\|^2- \|\H(g)\|^2-\left\la \p^* g, \G(\p^*
g)\right\ra-\|\G(\bp\p g)\|^2\\
&\leq & \|g\|^2,\ee since the Green operator is nonnegative. In
particular, if $\bp\p g=0$ and $g$ is $\p^*$-exact, we have
$\H(g)=0$ and $\p^* g=0$. Hence, we obtain the isometry $ \|\bp^*G\p
g\|=\|g\|$.  \eproof

\bproposition [$\overline{\partial}$-Inverse formula]\label{eq}
 Let $(E,h)$ be a Hermitian holomorphic vector bundle with semi-Nakano positive
curvature $\Theta^E$ over the compact K\"ahler manifold
$(X,\omega)$. Then, for any $g\in A^{n-1,\bullet}(X,E)$, $$s = {\bp
}^* \mathbb{G} \nabla'g$$ is a solution to the equation ${\bp } s =
\nabla'g$ with $\bp  \nabla'g =0$, such that $$\|s\|^2\leq \la
\nabla' g, \G  \nabla' g\ra.$$ This solution is unique as long as it
satisfies $\mathbb{H}(s)=0$ and $\overline{\partial}^{\ast}s=0$.
\end{proposition}

\begin{proof} By the well-known Bochner-Kodaira-Nakano identity $ \osqa= \sqa'
+ [\sqrt{-1}\Theta^E, \Lambda_\omega]$,  one can see that for any
$\phi\in A^{n,\bullet}(X,E)$,
$$\la \sqrt{-1}[\Theta^E, \Lambda_\omega] \phi,\phi\ra\geq 0$$
if $E$ is semi-Nakano positive( e.g. \cite{Demailly}). It
 implies  that,  for any $\phi\in A^{n,\bullet}(X,E),$ $$\langle \osqa
\phi,\phi\rangle\geq \langle\sqa'\phi,\phi \rangle. $$ Thus, on the
space $A^{n,\bullet}(X,E)$,
\begin{equation}\label{2ker}
\ker  \osqa\subseteq\ker \sqa'\ \textmd{and}\ (\ker
\sqa')^\perp\subseteq(\ker \osqa)^\perp.
\end{equation} By Hodge decomposition, we have
$$\bp  s = {\bp } {\bp }^* \mathbb{G} \nabla'g
= \nabla'g-\mathbb{H}\nabla'g-{\bp }^* {\bp } \mathbb{G} \nabla'g=
\nabla'g-\mathbb{H}\nabla'g= \nabla'g,$$ where the identity
$\mathbb{H}\nabla'g=0$ is used. Actually, we know $\nabla'g\bot\ker
\sqa'$ and obviously $\nabla'g\bot\ker \osqa$ by the first inclusion
of (\ref{2ker}).

The uniqueness of this solution follows easily.  In fact, if $s_1$
and $s_2$ are two solutions to $\bp s=\nabla' g$ with
$\H(s_1)=\H(s_2)=0$ and $\bp^*s_1=\bp^* s_2=0$, by setting
$\eta=s_1-s_2$, we see $\bp\eta =0$, $\H(\eta)=0$ and $\bp^*\eta=0$.
Therefore,
\begin{align*}
\eta=\mathbb{H}(\eta)+\osqa\mathbb{G}(\eta)
=\mathbb{H}(\eta)+(\overline
{\partial}\overline{\partial}^{*}+\overline{\partial}^{*}\overline{\partial
})\mathbb{G}(\eta) =0.
\end{align*}

\end{proof}

\section{Beltrami differentials and deformation theory}\label{Bel}

In this section we prove several new formulas to construct explicit
deformed differential operators for bundle valued differential forms
on the deformation spaces of K\"ahler manifolds. These formulas are
applied to the deformation spaces of CY manifolds in later sections
while more applications to the deformation theory of K\"ahler
manifolds and holomorphic line bundles will be discussed in the
sequel to this paper. Throughout this section, $X$ is always assumed
to be a complex manifold.

For $X_0\in \Gamma(X,T^{1,0}_X)$, the contraction operator is
defined as $$ i_{X_0}:\, A^{p,q}(X)\>A^{p-1,q}(X)$$ by
$$(i_{X_0}\omega)(X_1,\cdots, X_{p-1},Y_1,\cdots, Y_q)=\omega(X_0,X_1,\cdots, X_{p-1},Y_1,\cdots Y_q)$$
for $\omega\in A^{p,q}(X)$, $X_1,\cdots,X_{p-1}\in
\Gamma(X,T^{1,0}_X)$ and $Y_1,\cdots, Y_q\in \Gamma(X,T^{0,1}_X)$.
We will also use the notation `$\lrcorner$' to represent the
contraction operator in the sequel, that is,
$i_{X_0}(\omega)=X_0\lrcorner \omega$.

For $\phi\in A^{0,s}(X,T^{1,0}_X)$, the contraction operator can be
extended to \beq i_\phi:\,
A^{p,q}(X)\>A^{p-1,q+s}(X)\label{db3}.\eeq For example, if
$\phi=\eta\ts Y$ with $\eta\in A^{0,q}(X)$ and
$Y\in\Gamma(X,T^{1,0}_X)$, then for any $\om \in A^{p,q}(X),$ $$
(i_\phi)(\omega)=\eta\wedge (i_Y\omega).$$ The following result
follows easily. \blemma Let $\phi\in A^{0,q}(X,T^{1,0}_X)$ and
$\psi\in A^{0,s}(X,T^{1,0}_X)$. Then \beq i_\phi\circ
i_\psi=(-1)^{(q+1)(s+1)}i_{\psi}\circ i_{\phi}.\label{db2}\eeq
\elemma

For $Y\in \Gamma(X,T_X)$, the Lie derivative $\sL_Y$ is defined as
\beq \sL_Y=d\circ i_Y+i_Y\circ d :A^s(X)\>A^s(X). \eeq For any
$\phi\in A^{0,q}(X,T^{1,0}_X)$, we can define $i_\phi$ as
(\ref{db3}) and thus extend $\sL_\phi$ to be \beq \sL_\phi=(-1)^q
d\circ i_\phi+i_\phi\circ d. \eeq According to the types, we can
decompose
$$ \sL_\phi=\sL_\phi^{1,0}+\sL_{\phi}^{0,1},$$ where $$
\sL^{1,0}_\phi=(-1)^q\p\circ i_\phi+i_\phi\circ \p$$ and $$
\sL^{0,1}_\phi=(-1)^q\bp\circ i_\phi+i_\phi\circ \bp. $$
 Let
\[
\varphi^{i}=\frac{1}{p!}\sum\varphi^{i}_{\bar{j}_{1},\cdots,\bar{j}_{p}}%
d\bar{z}^{j_{1}}\wedge\cdots\wedge d\bar{z}^{j_{p}}\ts \p_i\
\textmd{and}\ \psi
^{i}=\frac{1}{q!}\sum\psi^{i}_{\bar{k}_{1},\cdots,\bar{k}_{q}}d\bar{z}^{k_{1}%
}\wedge\cdots\wedge d\bar{z}^{k_{q}}\ts \p_i.
\]
Then, we write
\begin{equation}
\label{Belt}[\varphi,\psi]=\sum_{i,j=1}^{n}(\varphi^{i}\wedge\partial_{i}%
\psi^{j}-(-1)^{pq}\psi^{i}\wedge\partial_{i}\varphi^{j})\otimes\partial_{j},
\end{equation}
where $$\partial_{i}\varphi^{j}=\frac{1}{p!}\sum\partial_{i}\varphi^{j}%
_{\bar{j}_{1},\cdots,\bar{j}_{p}}d\bar{z}^{j_{1}}\wedge\cdots\wedge
d\bar {z}^{j_{p}}$$ and similar for $\partial_{i}\psi^{j}$. In
particular, if $\varphi,\psi\in A^{0,1}(X,T^{1,0}_X)$,
$$[\varphi,\psi]=\sum
_{i,j=1}^{n}(\varphi^{i}\wedge\partial_{i}\psi^{j}+\psi^{i}\wedge\partial
_{i}\varphi^{j})\otimes\partial_{j}.$$

Let $(E,h)$ be a Hermitian holomorphic vector bundle over $X$ and
$\nabla$ be the Chern connection on $(E,h)$. Then the operators
$i_{\bullet}, \sL_{\bullet}$, $[\bullet,\bullet]$ can be extended to
any $E$-valued  $(p,q)$ form in the canonical way. For example, for
any $\phi\in A^{0,k}(X,T^{1,0}_X)$, on $A^{p,q}(X,E)$ we can define
\beq \sL_\phi=(-1)^k \nabla\circ i_\phi+i_\phi\circ \nabla.\eeq

 Then we have the following
general commutator formula.
\begin{lemma}[cf.\cite{[LR]}]\label{aaaa} For $\varphi\in A^{0,k}(X,T^{1,0}_X)$, $\varphi'\in
A^{0,k'}(X,T^{1,0}_X)$ and $\alpha\in A^{p,q}(X,E)$,
$$(-1)^{k'}\varphi\lc \mathcal{L}_{\varphi'}\alpha+(-1)^{k'k+1}\mathcal{L}_{\varphi'}(\varphi\lc
\alpha) =[\varphi,\varphi']\lc \alpha,$$ or equivalently,
$$[\mathcal{L}_{\varphi'},i_\varphi]=i_{[\varphi',\varphi]}.$$
 In
particular, if $\varphi,\varphi'\in A^{0,1}(X,T^{1,0}_X)$, then
\begin{equation}\label{f1}
[\varphi,\varphi']\lrcorner\alpha=-\nabla'(\varphi'\lrcorner(\varphi
\lrcorner\alpha))-\varphi\lrcorner(\varphi'\lrcorner\nabla'\alpha)
+\varphi\lrcorner\nabla'(\varphi'\lrcorner\alpha)+\varphi'
\lrcorner\nabla'(\varphi\lrcorner\alpha)
\end{equation}
and
\begin{equation}\label{f2}
0=-\db(\varphi'\lrcorner(\varphi
\lrcorner\alpha))-\varphi\lrcorner(\varphi' \lrcorner\db\alpha)
+\varphi\lrcorner\db(\varphi'\lrcorner\alpha)+\varphi'
\lrcorner\db(\varphi\lrcorner\alpha).
\end{equation}
\end{lemma}
\begin{proof} Since the formulas are all
local and $\mathbb C$-linear, without loss of generality, we can
assume that $$\varphi=\eta\ts \chi,\ \varphi'=\eta'\ts \chi',$$
where $\eta\in A^{0,k}(X)$, $\eta'\in A^{0,k'}(X)$, $\chi,\chi'\in
\Gamma(X,T^{1,0}_X)$ and $d\eta=d\eta'=0$. Since $d\eta=d\eta'=0$,
we have  $\chi'(\eta)=\chi(\eta')=0$. Hence, we obtain \be
[\varphi,\varphi']=\eta\wedge \eta' [\chi,\chi'].\ee On the other
hand, for any $\alpha\in A^{p,q}(X,E),$
\begin{align*} \mathcal{L}_{\varphi}\alpha
 &=\eta\wedge(\chi\lrcorner \nabla\alpha)+(-1)^k \nabla(\eta\wedge(\chi\lrcorner\alpha))\\
 &=\eta\wedge(\chi\lrcorner \nabla\alpha)+(-1)^k (d\eta\wedge(\chi\lrcorner\alpha)
 +(-1)^k\eta\wedge\nabla(\chi\lrcorner\alpha))\\
 &=\eta\wedge(\chi\lrcorner \nabla\alpha+\nabla(\chi\lrcorner\alpha))\\
 &=\eta\wedge \mathcal{L}_{\chi}\alpha.
\end{align*}
Now, we have
\begin{align*}
   \varphi\lc \mathcal{L}_{\varphi'}\alpha
 =&\eta\wedge\chi\lc(\eta'\wedge \mathcal{L}_{\chi'}\alpha)\\
 =&(-1)^{k'}\eta\wedge\eta'(\chi\lc \mathcal{L}_{\chi'}\alpha)\\
 =&(-1)^{k'}\eta\wedge\eta'\left(\mathcal{L}_{\chi'}(\chi\lc\alpha)-[\chi',\chi]\lc\alpha\right)\\
 =&(-1)^{k'}\left(\eta\wedge \mathcal{L}_{\varphi'}(\chi\lc\alpha)-\eta\wedge\eta'\wedge([\chi',\chi]\lc\alpha)\right)\\
 =&(-1)^{k'}[\varphi,\varphi']\lc\alpha+(-1)^{k'(1+k)}\mathcal{L}_{\varphi'}(\eta\wedge(\chi\lc\alpha))\\
 =&(-1)^{k'}[\varphi,\varphi']\lc\alpha+(-1)^{k'(1+k)}\mathcal{L}_{\varphi'}(\varphi\lc\alpha),
\end{align*}
where we apply the formula
$$[\chi',\chi]\lc\alpha=\mathcal{L}_{\chi'}(\chi\lc \alpha)-\chi\lc
\mathcal{L}_{\chi'}\alpha,$$ which is proven in \cite{[LR]}, and
$$\mathcal{L}_{\varphi'}(\varphi\lc\alpha)=(-1)^{k'k}\eta\wedge \mathcal{L}_{\varphi'}(\chi\lc\alpha).$$
In fact,
\begin{align*}
   &\mathcal{L}_{\varphi'}(\varphi\lc\alpha)\\
 =&\mathcal{L}_{\varphi'}(\eta\wedge(\chi\lc\alpha))\\
 =&\varphi'\lrcorner \nabla(\eta\wedge(\chi\lc\alpha))+(-1)^{k'}\nabla\circ\varphi\lrcorner(\eta\wedge(\chi\lc\alpha))\\
 =&\varphi'\lrcorner(d\eta\wedge(\chi\lc\alpha))+(-1)^{k}\varphi'\lrcorner(\eta\wedge\nabla(\chi\lc\alpha))
 +(-1)^{k'+k(k'-1)}\nabla(\eta\wedge(\varphi'\lrcorner(\chi\lc\alpha)))\\
 =&(-1)^{k+k(k'-1)}\eta\wedge(\varphi'\lrcorner(\nabla(\chi\lc\alpha)))
 +(-1)^{k'+k(k'-1)+k}\eta\wedge\nabla(\varphi'\lrcorner(\chi\lc\alpha))\\
 =&(-1)^{k'k}\eta\wedge \mathcal{L}_{\varphi'}(\chi\lc\alpha).
\end{align*}
\end{proof}

As an easy corollary, we have the following result which was known
as Tian-Todorov lemma.

\begin{lemma}[\cite{[To89],[T]}]
\label{TT3} If $\varphi,\psi\in A^{0,1}(X,T^{1,0}_X)$ and $\Omega\in
A^{n,0}(X)$, then one has
\begin{equation}
\label{TT}[\varphi,\psi]\lrcorner\Omega=-\partial(\psi\lrcorner(\varphi
\lrcorner\Omega))
+\varphi\lrcorner\partial(\psi\lrcorner\Omega)+\psi
\lrcorner\partial(\varphi\lrcorner\Omega).
\end{equation}
In particular, if $X$ is a CY manifold and $\Omega_0$ is a
nontrivial holomorphic $(n,0)$ form on $X$, then or any
$\varphi,\psi\in\mathbb{H}^{0,1}(X,T^{1,0}_X)$,
\begin{equation}
\label{TTCY}[\varphi,\psi]\lrcorner\Omega_{0}=-\partial(\psi\lrcorner
(\varphi\lrcorner\Omega_{0})).
\end{equation}
Note that, here both $\varphi\lrcorner\Omega_{0}$ and $\psi
\lrcorner\Omega_{0}$ are  harmonic.

\end{lemma}

Let $\phi\in A^{0,1}(X,T^{1,0}_X)$ and $i_\phi$ be the contraction
operator. Define an operator
$$e^{i_\phi}=\sum_{k=0}^\infty \frac{1}{k!} i_\phi^{k},$$
where $i_\phi^k=\underbrace{i_\phi\circ \cdots\circ i_\phi}_{k\
\text{copies}}$. Since the dimension of $X$ is finite, the summation
in the above formulation is also  finite.

The following theorem gives explicit formulas for the deformed
differential operators on the deformation spaces of complex
structures. It also explains why it is relatively easy to construct
extension of sections of the bundle $K_X+E$ where $K_X$ is the
canonical bundle of $X$. We remark that this result is motivated by
\cite{C} where a special case was proved.

\begin{theorem}
Let $\phi\in A^{0,1}(X,T^{1,0}_X)$. Then on the space
{$A^{\bullet,\bullet}(E)$}, we have \begin{equation}
e^{-i_{\phi}}\circ \nabla\circ e^{i_\phi}=\nabla-\mathcal
L_{\phi}-i_{\frac{1}{2}[\phi,\phi]},
\end{equation} or equivalently\begin{equation} e^{-i_{\phi}}\circ
\overline{\partial}\circ e^{i_\phi}=\overline{\partial}-\mathcal
L_{\phi}^{0,1}\label{f3}\end{equation} and
\begin{equation} e^{-i_{\phi}}\circ \nabla' \circ e^{i_\phi}=\nabla'-\mathcal
L^{1,0}_{\phi}- i_{\frac{1}{2}[\phi,\phi]}.\label{f4}\end{equation}
Moreover, if $\bp\phi=\frac{1}{2}[\phi,\phi]$, then \beq
\bp-\sL_{\phi}^{1,0}=e^{-i_\phi}\circ (\bp-\sL_\phi)\circ
e^{i_\phi}.\label{f35}\eeq
\end{theorem}

\begin{proof} (\ref{f3}) follows from (\ref{f1}) and formula
$$[\overline{\partial}, i_\phi^k]=k
i_\phi^{k-1}\circ[\overline{\partial},i_\phi],$$ which can be proved
by induction by using (\ref{f1}). Similarly, (\ref{f4}) follows from
(\ref{f2}) and \begin{equation} [\nabla', i_\phi^k]=k
i_{\phi}^{k-1}\circ [\nabla', i_\phi]-\frac{k(k-1)}{2}
i_\phi^{k-2}\circ i_{[\phi,\phi]},\quad k\geq2.
\label{p1}\end{equation} Now we prove (\ref{p1}) by induction. It is
obvious that (\ref{p1}) is equivalent to the statement that, for any
$k\geq 2$,
\begin{equation}\label{fk} F_k:=-ki_\phi^{k-1}\circ
\nabla'\circ i_\phi+(k-1)i_\phi^k\circ \nabla'+\nabla'\circ
i_\phi^k+\frac{k(k-1)}{2}
i_\phi^{k-2}i_{[\phi,\phi]}=0.\end{equation}

If $k=2$, it is (\ref{f2}). As for $k=3$,
\begin{align*}
0
 &=i_{[\phi,\phi]}\circ i_\phi-i_\phi\circ i_{[\phi,\phi]}\\
 &=3i_\phi\circ\nabla'\circ i^2_\phi-\nabla'\circ i_\phi^3-3i_\phi^2\circ\nabla'\circ i_\phi+i^3_\phi\circ\nabla'\\
 &=3i_\phi^2\circ\nabla'\circ i_\phi-2i_\phi^3\circ\nabla'-\nabla'\circ i_\phi^3-3i_\phi\circ i_{[\phi,\phi]}\\
 &=-F_3,
\end{align*}
where Lemma \ref{db2} is applied.

Now we assume that (\ref{fk}) is right for all integers less than
$k$ where $k\geq 4$. That is,
$$F_2=F_3=\cdots =F_{k-1}=0.
$$
We will show $F_k=0$. Now we set
\begin{align*}
G_k
 &=F_k-i_\phi\circ F_{k-1}\\
 &=-i_{\phi}^{k-1}\circ \nabla'\circ i_\phi
+i_\phi^k\circ \nabla'+\nabla'\circ i_\phi^k-i_\phi\circ
\nabla'\circ i_\phi^{k-1}+(k-1)i_\phi^{k-2}i_{[\phi,\phi]}.
\end{align*}
So, by induction, we have \begin{eqnarray*} &&G_k-i_\phi\circ
G_{k-1}\\
&=&\nabla'\circ i_\phi^k-2i_\phi\circ \nabla'\circ
i_\phi^{k-1}+i_{\phi}^2\circ \nabla'\circ i_\phi^{k-2}+
i_\phi^{k-2}\circ i_{[\phi,\phi]} \\
&=&(\nabla'\circ i_\phi^2+i_\phi^2\circ\nabla'-2i_\phi\circ
\nabla'\circ
i_\phi)\circ i_{\phi}^{k-2}+ i_\phi^{k-2}\circ i_{[\phi,\phi]}\\
&=&-i_{[\phi,\phi]}\circ i_{\phi}^{k-2}+
i_\phi^{k-2}\circ i_{[\phi,\phi]}\\
&=&-i_{\phi}\circ i_{[\phi,\phi]}\circ i_{\phi}^{k-3}+
i_\phi^{k-2}\circ i_{[\phi,\phi]}\\
&=&-i_{\phi}^2\circ i_{[\phi,\phi]}\circ i_{\phi}^{k-4}+
i_\phi^{k-2}\circ i_{[\phi,\phi]}\\
&=&-i_{\phi}^{k-3}\circ i_{[\phi,\phi]}\circ i_{\phi}+
i_\phi^{k-3}\circ i_{\phi}\circ i_{[\phi,\phi]}\\
&=&-i_{\phi}^{k-3}\circ (i_{[\phi,\phi]}\circ i_{\phi}-i_{\phi}\circ
i_{[\phi,\phi]}) \\
&=&0\end{eqnarray*} since $i_{[\phi,\phi]}i_\phi-i_\phi
i_{[\phi,\phi]}=0$. ( Alternatively, we can also approach this
equality directly by induction on the term $G_k-i_\phi\circ
G_{k-1}$, i.e., $0=G_{k-1}-i_\phi\circ G_{k-2}=-i_{[\phi,\phi]}\circ
i_{\phi}^{k-3}+ i_\phi^{k-3}\circ i_{[\phi,\phi]}$.) The proof of
(\ref{p1}) is finished. From (\ref{p1}), it follows that
$$[\nabla', e^{i_\phi}]=e^{i_\phi}\circ [\nabla',
i_\phi]-e^{i_\phi}\circ \frac{1}{2} i_{[\phi,\phi]}$$ by comparing
 degrees. Then, we have
\begin{eqnarray*} e^{-i_\phi}\circ \nabla'\circ e^{i_\phi}&=&e^{-i_\phi}\circ
[\nabla', e^{i_\phi}]+\nabla'
\\
&=&[\nabla', i_\phi]+\nabla'-
i_{\frac{1}{2}[\phi,\phi]}\\
&=&\nabla'-\mathcal L_{\phi}^{1,0}- i_{\frac{1}{2}[\phi,\phi]}.
\end{eqnarray*} Now we finish the proof of (\ref{f4}) while the proof of
(\ref{f3}) is similar.

Finally, when $\bp\phi=\frac{1}{2}[\phi,\phi]$, we have $
[2\bp-\sL_\phi, i_\phi]=0$ and thus $$[2\bp-\sL_\phi,
e^{i_\phi}]=0,$$ which implies that
$$ e^{-i_\phi}\circ (\bp-\sL_\phi)\circ e^{i_\phi}=2\bp-\sL_\phi
-e^{-i_\phi}\circ \bp\circ e^{i_\phi}=\bp-\sL_\phi^{1,0}.$$
\end{proof}

\bcorollary\label{key2} If $\sigma\in A^{n,\bullet}(X,E)$, we have
\be \left(e^{- i_\phi}\circ \nabla \circ
e^{ i_\phi}\right)(\sigma)&=&\bp\sigma-\sL^{1,0}_{\phi}(\sigma)+\ i_{\bp\phi-\frac{1}{2}[\phi,\phi]}(\sigma)\\
&=&\bp\sigma+\nabla'(\phi\lrcorner
\sigma)+\left(\bp\phi-\frac{1}{2}[\phi,\phi]\right)\lrcorner
\sigma.\ee In particular, if $\phi$ is integrable, i.e.,
$\bp\phi-\frac{1}{2}[\phi,\phi]=0$, then \beq \left(e^{-
i_\phi}\circ \nabla \circ e^{
i_\phi}\right)(\sigma)=\bp\sigma+\nabla'(\phi\lrcorner
\sigma)\label{rec1}.\eeq \ecorollary

 The above formula gives an explicit recursive formula to construct deformed cohomology
 classes for deformation of K\"ahler manifolds. When $E$ is a trivial bundle, the above formula was used in \cite{[GLT]} to
  study the global Torelli theorem.

%\vspace{1cm}
\section{Global canonical family of Beltrami differentials}
\label{gcfB} In this section, based on the techniques developed in
Sections \ref{dbeqn} and \ref{Bel}, we shall construct the following
globally convergent power series of Beltrami differentials in
$L^{2}$-norm on CY manifolds.  To avoid the bewildering notations,
we just present the details on the one-parameter case and then give
a sketch of the multi-parameter case.

The convergence of the power series in the following lemma is
crucial in our proof of the global convergence and regularity
results.
\begin{lemma}
\label{general1/2} Let $\{x_{i}\}_{i=1}^{+\infty}$ be a series given
by $$x_{k}:=c\sum_{i=1}^{k-1}x_{i}\cdot x_{k-i}, \qtq{$k\geq 2$}$$
inductively with real initial value $x_{1}$. Then the power series
$S(\tau)=\displaystyle\sum_{i=1}^{\infty}x_{i}\tau^{i}$ converges as
long as $|\tau|\leq\frac{1}{|4cx_{1}|}$.
\end{lemma}
\begin{proof}
Setting $S:=S(\tau)=\displaystyle\sum_{i=1}^{\infty}x_{i}\tau^{i}$,
we have
\begin{equation}
\label{S2S}cS^{2} =c\left( \displaystyle\sum_{i=1}^{\infty}x_{i}\tau
^{i}\right)  \left(
\displaystyle\sum_{j=1}^{\infty}x_{j}\tau^{j}\right)
=\sum_{k=1}^{+\infty} x_{k}\tau^{k}-x_{1}\tau=S-x_{1}\tau.
\end{equation}
It follows from (\ref{S2S}) that
\[
S=\frac{1\pm\sqrt{1-4cx_{1}\tau}}{2c}.
\]
Here we take $S(\tau)=\frac{1-\sqrt{1-4cx_{1}\tau}}{2c},$ since we
have $S(0)=0$ according to the assumption. Therefore, we  have the
following expansion for $S$
\begin{align*}
S  &  =\frac{1}{2c}\left(  1-\left(
1+\sum_{n\geq1}\frac{\frac{1}{2}(\frac
{1}{2}-1)\cdots(\frac{1}{2}-n+1)}{n!}(-{4cx_{1}}\tau)^{n}\right)  \right) \\
&  =\sum_{n\geq1}\frac{1}{2c}\left(  \frac{\frac{1}{2}(1-\frac{1}{2}%
)\cdots((n-1)-\frac{1}{2})}{n!}\right)  {(4cx_{1})}^{n}\tau^{n},
\end{align*}
which implies that
\[
x_{n}=\frac{\frac{1}{2}(1-\frac{1}{2})\cdots((n-1)-\frac{1}{2})}{{2c}%
n!}{(4cx_{1})}^{n}, \qtq{for} n\geq 2.
\]
This is the explicit expression for each $x_{n}$. Now it is easy to
check that the convergence radius of the power series
$S=\displaystyle\sum_{i=1}^{\infty }x_{i}\tau^{i}$ is
$(4|cx_{1}|)^{-1}$, and that this power series still converges when
$\tau=\pm\frac{1}{4|cx_{1}|}$.
\end{proof}

%\subsection{Global canonical family of Beltrami differentials}
%\label{Global convergence of the Beltrami Differential}
Now we prove the global convergence of the Beltrami differential
from the Kodaira-Spencer-Kuranishi theory. All sub-indices of the
Beltrami differentials are at least $1$.

The following result is contained in \cite{[To89],[T]}, we briefly
recall here for the reader's convenience.
\begin{lemma}
\label{closed} Assume that for $\varphi_{\nu}\in
A^{0,1}(X,T^{1,0}_X)$, $\nu=2,\cdots,K,$
\begin{equation}
\label{2K}\overline{\partial}\varphi_{\nu}=\frac{1}{2}%
\sum_{\alpha+\beta=\nu} \left[
\varphi_{\alpha},\varphi_{\beta}\right]\quad \emph{and}\quad
\overline{\partial}\varphi_{1}=0.
\end{equation}
Then one has
\begin{equation}
\overline{\partial}\left(\sum_{\nu+\gamma=K+1 }\left[
\varphi_{\nu},\varphi_{\gamma}\right]  \right)  =0.
\end{equation}
\end{lemma}

\begin{proof}
By definition formula (\ref{Belt}), one has
\begin{equation}
\label{cfb}[\overline{\partial}\varphi,\varphi']=-[\varphi',\overline{\partial}\varphi].
\end{equation}
Then we have
\begin{align*}
\frac{1}{2}\overline{\partial}\left( \sum_{\nu+\gamma=K+1 }\left[
\varphi_{\nu},\varphi_{\gamma}\right]\right) =  &
\frac{1}{2}\sum_{\nu+\gamma=K+1 } \left( \left[
\overline{\partial}\varphi_{\nu},\varphi_{\gamma}\right]  -\left[  \varphi_{\nu},\overline{\partial}\varphi_{\gamma}\right]  \right) \\
=  &  \sum_{\nu+\gamma=K+1 }\left[ \overline
{\partial}\varphi_{\nu},\varphi_{\gamma}\right] \\
=  &  \frac{1}{2}\sum_{\substack{\nu+\gamma=K+1}} \left[
\sum_{\substack{\alpha+\beta=\nu}}\left[
\varphi_{\alpha},\varphi_{\beta}\right]
,\varphi_{\gamma}\right] \\
=  & \frac{1}{2}\sum_{\substack{\alpha+\beta+\gamma=K+1}} \left[
\left[  \varphi_{\alpha},\varphi_{\beta}\right]
,\varphi_{\gamma}\right]  ,
\end{align*}
where the second equality is implied by (\ref{cfb}) and the third
one follows from the assumption (\ref{2K}). When
$\alpha=\beta=\gamma$, by Jacobi identity one has
\[
3\left[  \left[  \varphi_{\alpha},\varphi_{\beta}\right]
,\varphi_{\gamma}\right]  =0.
\]
Otherwise, Jacobi identity implies that
\[
\left[  \left[  \varphi_{\alpha},\varphi_{\beta}\right]
,\varphi_{\gamma}\right] +\left[  \left[
\varphi_{\beta},\varphi_{\gamma}\right] ,\varphi_{\alpha}\right]
+\left[  \left[ \varphi_{\gamma},\varphi_{\alpha}\right]
,\varphi_{\beta}\right]  =0.
\]

\end{proof}

We need some basic estimates.  At
 first, let's recall the following estimate in \cite[p.162]{[MK]}, for any $\eta_1,\eta_2\in A^{0,1}(X,T_X^{1,0})$,
\beq
\left\|\frac{1}{2}\bp^*G[\eta_1,\eta_2]\right\|_{\mathscr{C}^1}\leq
C_1 \|\eta_1\|_{\mathscr{C}^1}\cdot
\|\eta_2\|_{\mathscr{C}^1}\label{key101} \eeq where $C_1$ is a
constant independent of $\eta_1, \eta_2$.
 Next, for any $(n,0)$-from
$s$ on $X$, we have \beq \|\eta_1\lrcorner s\|_{L^2}\leq
\|\eta_1\|_{\mathscr{C}^0}\cdot \|s\|_{L^2}\leq
\|\eta_1\|_{\mathscr{C}^1}\cdot \|s\|_{L^2}.\label{e100}\eeq This
inequality follows by checking the local inner product by
definition.
 Similarly,  \beq
\|\eta_1\lrcorner\eta_2\lrcorner s\|_{L^2}\leq C_2
\|\eta_1\|_{{\mathscr C}^1}\cdot \|\eta_2\|_{\mathscr{C}^1}\cdot
\|s\|_{L^2}.\label{e200}\eeq where $C_2$ is independent of
$\eta_1,\eta_2, s$.

\begin{theorem}
\label{Phi} Let $X$ be a CY manifold and $\varphi_{1}\in\mathbb{H}
^{0,1}(X,T^{1,0}_X)$ with norm
$\|\varphi_{1}\|_{\mathscr{C}^1}=\frac{1}{4C_1}$. Then for any
nontrivial holomorphic $(n,0)$ form $\Omega_0$ on $X$, there exits a
smooth globally convergent power series  for $|t|<1$,
\begin{equation}
\label{Phips}\Phi(t)=\varphi_{1}t^{1}+\varphi_{2}t^{2}+\cdots+\varphi_{k}%
t^{k}+\cdots\in A^{0,1}(X,T^{1,0}_X),
\end{equation}
which satisfies:

$a)$ $\overline{\partial}\Phi(t)=\frac{1}{2}[\Phi(t),\Phi(t)]$;

$b)$ $\overline{\partial}^{*}\varphi_{k}=0$\ {for each $k\geq1$};

$c)$ $\varphi_{k}\lrcorner\Omega_{0}$ is $\partial$-exact\ for each
$k\geq2$;

$d)$ $\|\Phi(t)\lrcorner\Omega_{0}\|_{L^{2}}<\infty$ as long as
$|t|<1$.
\end{theorem}

\begin{proof}
Let us first review the construction of the power series $\Phi(t)$
by induction from \cite{[T]} and \cite{[To89]}. Suppose that we have
constructed $\varphi_{k}$ for $2\leq k\leq j$ such that:

$a)$
$\overline{\partial}\varphi_{k}=\frac{1}{2}\sum_{i=1}^{k-1}[\varphi
_{k-i},\varphi_{i}]$;

$b)$ $\overline{\partial}^{*}\varphi_{k}=0$;

$c)$ $\varphi_{k}\lrcorner\Omega_{0}$ is $\partial$-exact and thus
$\partial(\varphi_{k}\lrcorner\Omega_{0})=0$.\newline Then we need
to construct $\varphi_{j+1}$ such that:

$a^{\prime})$ $\overline{\partial}\varphi_{j+1}=\frac{1}{2}\sum_{i=1}%
^{j}[\varphi_{j+1-i},\varphi_{i}]$;

$b^{\prime})$ $\overline{\partial}^{*}\varphi_{j+1}=0$;

$c^{\prime})$ $\varphi_{j+1}\lrcorner\Omega_{0}$ is $\partial$-exact
and thus $\partial(\varphi_{j+1}\lrcorner\Omega_{0})=0$.\newline
Actually, it follows from Lemma \ref{TT3} and the assumption $c)$
that
\begin{equation}
\sum_{i=1}^{j}[\varphi_{j+1-i},\varphi_{i}]\lrcorner\Omega_{0}=-\partial
\left(  \sum_{i+k=j+1}\varphi_{i}\lrcorner\varphi_{k}\lrcorner\Omega
_{0}\right)  .\label{key}
\end{equation}
Then, Lemma \ref{closed} and the assumption $a)$ imply
\begin{equation}
\label{dbar}\overline{\partial}\partial\left(  \sum_{i+k=j+1}\varphi
_{i}\lrcorner\varphi_{k}\lrcorner\Omega_{0}\right)
=\overline{\partial }\left(
\sum_{i=1}^{j}[\varphi_{j+1-i},\varphi_{i}]\right)  \lrcorner
\Omega_{0}=0.
\end{equation}
So formula (\ref{dbar}) and Proposition \ref{1eq} tell us that the
equation
$$\overline{\partial}\Psi_{j+1}=-\partial\left( \sum
_{i+k=j+1}\varphi_{i}\lrcorner\varphi_{k}\lrcorner\Omega_{0}\right)
$$ has a solution
$\Psi_{j+1}=-\overline{\partial}^{*}\mathbb{G}\partial\left(
\sum_{i+k=j+1}\varphi_{i}\lrcorner\varphi_{k}\lrcorner\Omega_{0}\right)
$. Hence, we define
\[
\varphi_{j+1}=\frac{1}{2}\Psi_{j+1}\lrcorner\Omega_{0}^{*},
\]
where $\Omega_{0}^{*}:=\frac{\partial}{\partial
z^{1}}\wedge\cdots\wedge \frac{\partial}{\partial z^{n}}$ in local
coordinates is the dual of $\Omega_0$. It is easy to check that
\[
\overline{\partial}^{*}(\Psi_{j+1}\lrcorner\Omega_{0}^{*})=\overline{\partial
}^{*}(\Psi_{j+1})\lrcorner\Omega_{0}^{*}
+\Psi_{j+1}\lrcorner\overline {\partial}^{*}\Omega_{0}^{*}=0,
\]
since $\Omega_{0}$ is parallel, and also $\overline{\partial}\varphi
_{j+1}=\frac{1}{2}\sum_{i=1}^{j}[\varphi_{j+1-i},\varphi_{i}]$. See
\cite[Lemma~1.2.2]{[To89]} for more details. Now we have
completed the construction of $\varphi_{j+1}=\frac{1}{2}\Psi_{j+1}%
\lrcorner\Omega_{0}^{*},$ which is shown to satisfy Properties
$a^{\prime})$, $b^{\prime})$ and $c^{\prime})$. To complete this
induction, it suffices to work out the case $j=2$. It is obvious
that $\varphi_{2}$ can be constructed as
\[
\varphi_{2}=\frac{1}{2}\overline{\partial}^{*}\mathbb{G}\partial\left(
\varphi_{1}\lrcorner\varphi_{1}\lrcorner\Omega_{0}\right)
\lrcorner\Omega _{0}^{*},
\]
which satisfies $a)$, $b)$ and $c)$. Moreover, one has the following
equality for each $k\geq2$,
\begin{equation}
\varphi_{k}\lrcorner\Omega_{0}=\frac{1}{2}\overline{\partial}^{*}%
\mathbb{G}\partial\sum_{\substack{i+j=k\geq2}}\varphi_{i}\lrcorner\varphi
_{j}\lrcorner\Omega_{0}.
\end{equation}

Next, let us prove the $L^{2}$-convergence and regularity of
$\Phi(t)$.
 Without loss of
generality we can assume  $\|\Omega _{0}\|_{L^2}=1$  and thus have
for $|t|<1$,
\begin{align*}
\|\Phi(t)\lrcorner\Omega_{0}\|_{L^2}  &  =\left\|  (\varphi_{1}%
\lrcorner\Omega_{0})t+(\varphi_{2}\lrcorner\Omega_{0})t^{2}+\cdots
+(\varphi_{k}\lrcorner\Omega_{0})t^{k}+\cdots\right\|_{L^2} \\
&  =\left\|  (\varphi_{1}\lrcorner\Omega_{0})t+\sum_{j=2}^{\infty}\frac{1}%
{2}\overline{\partial}^{*}\mathbb{G}\partial\left(
\sum_{i+k=j}\varphi
_{i}\lrcorner\varphi_{k}\lrcorner\Omega_{0}\right)  t^{j}\right\|_{L^2} \\
\qtq{(Theorem \ref{1main})}&  \leq \frac
{1}{4C_1}|t|+\sum_{j=2}^{\infty}\frac{1}{2}\left(
\sum_{i+k=j}\left\| \varphi
_{i}\lrcorner\varphi_{k}\lrcorner\Omega_{0}\right\|_{L^2} \right)
|t|^{j}\\
\qtq{(Using (\ref{e200}))}&\leq  \frac{1}{4C_1}|t|+
\sum_{j=2}^{\infty}\frac{C_2}{2} \sum_{i+k=j}\left(\left\| \varphi
_{i}\right\|_{\mathscr{C}^1}\cdot\left\|\varphi_{k}\|_{\mathscr{C}^1}\cdot
\|\Omega_{0}\right\|_{L^2} \right) |t|^{j}\\
&\leq  \frac{1}{4C_1}|t|+ \sum_{j=2}^{\infty}\frac{C_2}{2}
\sum_{i+k=j}\left(\left\| \varphi
_{i}\right\|_{\mathscr{C}^1}\cdot\|\varphi_{k}\|_{\mathscr{C}^1}\right)
|t|^{j}.
\end{align*}
Now we set a sequence $\{x_j\}$ as in  Lemma \ref{general1/2}:
$$x_1=\frac{1}{4C_1}, \qtq{and} x_j:=C_1\sum_{i+k=j} x_i\cdot x_k, \qtq{for} j\geq 2.$$
Therefore by Lemma \ref{general1/2}, $\displaystyle
\sum_{j=1}^\infty x_j t^j$ has convergent radius
$$\frac{1}{4C_1|x_1|}=1.$$ Next, we claim \beq
\|\varphi_j\|_{\mathscr{C}^1}\leq  x_j \qtq{for } j=1,
2,\cdots.\label{key100}\eeq  By assuming (\ref{key100}), we have \be
\|\Phi(t)\lrcorner\Omega_{0}\|_{L^2}&\leq&  \frac{1}{4C_1}|t|+
\sum_{j=2}^{\infty}\frac{C_2}{2} \sum_{i+k=j}\left(\left\| \varphi
_{i}\right\|_{\mathscr{C}^1}\cdot\|\varphi_{k}\|_{\mathscr{C}^1}\right)
|t|^{j}\\&\leq& \frac{1}{4C_1}|t|+ \sum_{j=2}^{\infty}\frac{C_2}{2}
\sum_{i+k=j}\left(x_i\cdot x_k\right)
|t|^{j}\\
&\leq & \frac{1}{4C_1}|t|+
\frac{C_2}{2C_1}\sum_{j=2}^{\infty}x_j|t|^{j}\\
&\leq &  \frac{1}{4C_1}|t|-\frac{C_2}{8C_1^2}|t|+
\frac{C_2}{2C_1}\sum_{j=1}^{\infty}x_j|t|^{j}<\infty \ee  for
$|t|<1$ by Lemma \ref{general1/2}. In the following we shall prove
(\ref{key100}) by induction. \noindent From the iteration relation,
$$\overline{\partial}\varphi_{k}=\frac{1}{2}\sum_{i=1}^{k-1}[\varphi
_{k-i},\varphi_{i}],$$ we see
$\bp\varphi_2=\frac{1}{2}[\varphi_1,\varphi_1]$, or equivalently,
$$\varphi_2=\frac{1}{2}\bp^*G[\varphi_1,\varphi_1].$$
Hence, by (\ref{key101}), we get
$$\|\varphi_2\|_{\mathscr{C}^1}\leq C_1 \|\varphi_1\|_{\mathscr{C}^1}\cdot \|\varphi_1\|_{\mathscr{C}^1}\leq C_1 x_1\cdot x_1= x_2$$
since $x_1=\|\varphi_1\|_{\mathscr{C}^1}$. By induction, we assume
$$\|\varphi_j\|_{\mathscr{C}^1}\leq  x_j \qtq{for $j=1,\cdots, k-1$}.$$
and we shall prove $\|\varphi_k\|_{\mathscr{C}^1}\leq  x_k$. In
fact, we have
$$\varphi_k=\frac{1}{2}\bp^*G\left(\sum_{i=1}^{k-1}[\phi_{k-i},\varphi_i]\right),$$
and so by (\ref{key101}) and induction conditions, \be
\|\varphi_k\|_{\mathscr{C}^1}&\leq&
C_1\sum_{i=1}^{k-1}\|\varphi_{k-i}\|_{\mathscr{C}^1}\cdot
\|\varphi_i\|_{\mathscr{C}^1}\\
&\leq & C_1  \sum_{i=1}^{k-1} x_{k-i}\cdot x_i= x_k.\ee Hence, we
complete the proof  of $(\ref{key100})$.

For local regularity of $\Phi(t)$(i.e., $t$ sufficiently small)  it
follows from standard elliptic operator theory (e.g.\cite{[MK]}).
But for global regularity( $|t|< 1$), their proof does not work
directly. Here we use a different approach to prove it. At first, we
see that $\Phi(t)\lrcorner \Omega_0$ is $\p$-closed in the
distribution sense, i.e. \beq \p (\Phi(t)\lrcorner \Omega_0)=0,
\qtq{in the distribution sense}\label{hol}\eeq by using the
definition of $\Phi$ and the fact that
$\varphi_{k}\lrcorner\Omega_{0}$ are all $\partial$-exact for $k\geq
2$, $\varphi_1\lrcorner \Omega_0$ is harmonic. In fact, for any test
form $\eta$ on $X$, \be (\Phi(t)\lrcorner \Omega_0,
\p^*\eta)&=&\lim_{k\>\infty}\left(\left(\sum_{i=1}^k\varphi_i
t^i\right)\lrcorner \Omega_0,
\p^*\eta\right)=\lim_{k\>\infty}\left(\sum_{i=1}^k\p
(\varphi_i\lrcorner\Omega_0),\eta\right)=0.\ee
 Since $e^{\Phi(t)}\lrcorner\Omega_0$ is a family of
$(n,0)$ forms on $X_t$, by Corollary \ref{key2}( for more complete
argument, see Proposition \ref{holcriteria}), we obtain \beq
\bp_t\left(e^{\Phi(t)}\lrcorner \Omega_0\right)=0 \qtq{in the
distribution sense}\label{key4}\eeq where $\bp_t$ is the
$(0,1)$-part of the differential operator $d$ on $X_t$ induced by
the complex structure $J_{\Phi(t)}$. Therefore, by the
hypoellipticity of $\bp_t$ on $(n,0)$ forms, we obtain
$e^{\Phi(t)}\lrcorner \Omega_0$ is a holomorphic $(n,0)$ form on
$X_t$ and so $e^{\Phi(t)}\lrcorner \Omega_0$ is smooth on $X_t$ and
so on $X$. Finally, by contracting $\Om_0^*$ as above,  we obtain
that $e^{\Phi(t)}$ is smooth on $X$, and so is $\Phi(t)$.
\end{proof}

Now we state the following multi-parameter result, while we just
sketch its proof since it is essentially the same as the
one-parameter case.

\begin{theorem}\label{mPhi}
Let $X$ be a CY manifold and $\{\varphi_{1}, \cdots, \varphi
_{N}\}\in\mathbb{H}^{0,1}(X,T^{1,0}_X)$ be a basis with norm
$\|\varphi _{i}\|_{\mathscr{C}^1}=\frac{1}{8NC_1}$. Then  for any
nontrivial holomorphic $(n,0)$ form $\Omega_0$ on $X$, and $|t|<1$,
we can construct a smooth power series of Beltrami differentials on
$X$ as follows
\begin{equation}
\Phi(t)=\sum_{|I|\geq1}\varphi_{I}t^{I}
=\sum_{\substack{\nu_{1}+\cdots +\nu_{N}\geq1,\\\text{each
$\nu_{i}\geq0, i=1,2,\cdots$}}}\varphi_{\nu
_{1}\cdots\nu_{N}}t^{\nu_{1}}_{1}\cdots t^{\nu_{N}}_{N}\in A^{0,1}%
(X,T^{1,0}_X),
\end{equation}
where $\varphi_{0\cdots\nu_{i}\cdots0}=\varphi_{i}$. This power
series has the following properties:

$a)$ $\overline{\partial}\Phi(t)=\frac{1}{2}[\Phi(t),\Phi(t)]$, the
integrability condition;

$b)$ $\overline{\partial}^{*}\varphi_{I}=0$ for each multi-index $I$
with $|I|\geq1$;

$c)$ $\varphi_{I}\lrcorner\Omega_{0}$ is $\partial$-exact for each
$I$ with $|I|\geq2$. and more importantly,

$d)$ global convergence: $\|\Phi(t)\lrcorner\Omega_{0}\|\leq\sum_{I}%
\|\varphi_{I}\lrcorner\Omega_{0}\|\cdot|t|^{|I|}<\infty$ as long as
$|t|<1$.
\end{theorem}
\begin{proof} Let us construct the power series $\Phi(t)$ in multi-parameters by
induction. Write
$$
\mathcal{B}_{\gtreqqless K}=\{\varphi_{\nu_{1}\cdots\nu_{N}}\in A^{0,1}%
(M,T^{1,0}_{M})\ |\ \textmd{each integer}\ \nu_{i}\geq0\ {\textmd{and}}%
\ \nu_{1}+\cdots+\nu_{N}\gtreqqless K,\ K\geq1\}.
$$
It is easy to see that $\Phi(t)$ should satisfy:

$a)$ $\overline{\partial}\varphi_{\nu_{1}\cdots\nu_{N}}=\frac{1}{2}%
\sum\limits_{\substack{\alpha_{i}+\beta_{i}=\nu_{i}}} \left[ \varphi
_{\alpha_{1}\cdots\alpha_{N}},\varphi_{\beta_{1}\cdots\beta_{N}}\right]
$ for $\varphi_{\nu_{1}\cdots\nu_{N}}\in\mathcal{B}_{\geq2}$;

$b)$ $\overline{\partial}^{*}\varphi_{\nu_{1}\cdots\nu_{N}}=0$ for
$\varphi_{\nu_{1}\cdots\nu_{N}}\in\mathcal{B}_{\geq1}$;

$c)$ $\varphi_{\nu_{1}\cdots\nu_{N}}\lrcorner\Omega_{0}$ is
$\partial$-exact and thus
$\partial(\varphi_{\nu_{1}\cdots\nu_{N}}\lrcorner\Omega_{0})=0$ for
each $\varphi_{\nu_{1}\cdots\nu_{N}}\in\mathcal{B}_{\geq2}$.

Assuming that the above three assumptions hold for
$\varphi_{\nu_{1}\cdots
\nu_{N}}\in\mathcal{B}_{\geq2}\cap\mathcal{B}_{\leq K}$, then one
can construct $\varphi_{\nu_{1}\cdots\nu_{N}}\in\mathcal{B}_{K+1}$
such that it also satisfies these three assumptions. In fact, Lemma
\ref{TT3} and the
assumption $c)$ for $\varphi_{\nu_{1}\cdots\nu_{N}}\in\mathcal{B}_{\geq2}%
\cap\mathcal{B}_{\leq K}$ imply that
\begin{equation}
[\varphi_{\alpha_{1}\cdots\alpha_{N}},\varphi_{\beta_{1}\cdots\beta_{N}%
}]\lrcorner\Omega_{0}= -\partial\left(  \varphi_{\alpha_{1}\cdots\alpha_{N}%
}\lrcorner\varphi_{\beta_{1}\cdots\beta_{N}}\lrcorner\Omega_{0}\right)
,
\end{equation}
where $\sum_{i}\alpha_{i}+\sum_{j}\beta_{j}=K+1$. Then, by
multi-index Lemma \ref{closed}
and the assumption $a)$ for $\varphi_{\nu_{1}\cdots\nu_{N}}\in\mathcal{B}%
_{\geq2}\cap\mathcal{B}_{\leq K}$, we have
\begin{equation}
\label{dbar0}\overline{\partial}\partial\left(
\sum_{\substack{\alpha
_{i}+\beta_{i}=\nu_{i}}}\varphi_{\alpha_{1}\cdots\alpha_{N}}\lrcorner
\varphi_{\beta_{1}\cdots\beta_{N}}\lrcorner\Omega_{0}\right)
=\overline {\partial}\left(
\sum_{\substack{\alpha_{i}+\beta_{i}=\nu_{i}}}[\varphi
_{\alpha_{1}\cdots\alpha_{N}},\varphi_{\beta_{1}\cdots\beta_{N}}]\right)
\lrcorner\Omega_{0}=0,
\end{equation}
for any $\varphi_{\nu_{1}\cdots\nu_{N}}\in\mathcal{B}_{ K+1}$.
Therefore, one
can construct $\Psi_{\nu_{1}\cdots\nu_{N}}$ directly by $\overline{\partial}%
$-Inverse formula \ref{eq} and (\ref{dbar0}) as
\[
\Psi_{\nu_{1}\cdots\nu_{N}}=-\overline{\partial}^{*}\mathbb{G}\partial\left(
\sum\limits_{\substack{\alpha_{i}+\beta_{i}=\nu_{i}}}
\varphi_{\alpha
_{1}\cdots\alpha_{N}}\lrcorner\varphi_{\beta_{1}\cdots\beta_{N}}%
\lrcorner\Omega_{0}\right)  .
\]

Hence we define
\[
\varphi_{\nu_{1}\cdots\nu_{N}}=\frac{1}{2}\Psi_{\nu_{1}\cdots\nu_{N}}%
\lrcorner\Omega_{0}^{*}\in\mathcal{B}_{K+1},
\]
where $\Omega_{0}^{*}:=\frac{\partial}{\partial
z^{1}}\wedge\cdots\wedge \frac{\partial}{\partial z^{n}}$ is the
dual of $\Omega_{0}$. Then it is easy to check that
\[
\overline{\partial}^{*}(\Psi_{\nu_{1}\cdots\nu_{N}}\lrcorner\Omega_{0}%
^{*})=\overline{\partial}^{*}(\Psi_{\nu_{1}\cdots\nu_{N}})\lrcorner\Omega
_{0}^{*} +\Psi_{\nu_{1}\cdots\nu_{N}}\lrcorner\overline{\partial}^{*}%
\Omega_{0}^{*}=0
\]
since $\Omega_{0}$ is parallel, and also
$\overline{\partial}\varphi_{\nu
_{1}\cdots\nu_{N}}=\frac{1}{2}\sum\limits_{\substack{\alpha_{i}+\beta_{i}%
=\nu_{i}}}\left[ \varphi_{\alpha_{1}\cdots\alpha_{N}},\varphi_{\beta
_{1}\cdots\beta_{N}}\right]  $. To complete this induction, we
construct $\varphi_{\nu_{1}\cdots\nu_{N}}\in\mathcal{B}_{2}$ as
\begin{equation}
\label{varphip}\varphi_{\nu_{1}\cdots\nu_{N}}=
\begin{cases}
-\overline{\partial}^{*}\mathbb{G}\partial\left(
\varphi_{i}\lrcorner \varphi_{j}\lrcorner\Omega_{0}\right)
\lrcorner\Omega_{0}^{*}, \qquad\text{if
$\nu_{i}=\nu_{j}=1$, $i\neq j$},\\
-\frac{1}{2}\overline{\partial}^{*}\mathbb{G}\partial\left(  \varphi
_{i}\lrcorner\varphi_{i}\lrcorner\Omega_{0}\right)  \lrcorner\Omega_{0}%
^{*},\qquad\text{if $\nu_{i}=2$, for some $i\in\{1,\cdots,N\}$},
\end{cases}
\end{equation}
which obviously satisfies $a)$, $b)$ and $c)$.

Up to now we have completed the construction of the power series
$\Phi(t)$ satisfying $a)$, $b)$ and $c)$ as in Theorem \ref{Phi}. By
using similar arguments as in the proof of  Theorem \ref{Phi}, we
get the global convergence in $L^{2}$-norm and also the smoothness
of $\Phi(t)$ .
\end{proof}

\section{Global canonical family of holomorphic $(n,0)$-forms}

\label{gcfh} Based on the construction of $L^{2}$-global canonical
family $\Phi(t)$ of Beltrami differentials in Theorem \ref{mPhi}, we
can construct an {$L^{2}$-global canonical family of holomorphic
$(n,0)$-forms on the deformation spaces of CY manifolds. By using a
similar method, we can also construct {$L^{2}$-global canonical
family of holomorphic $(n,0)$-forms on the deformation spaces of
general compact K\"ahler manifolds.

\subsection{Global canonical family on Calabi-Yau manifolds}
\label{iteration} Let $X$ be an $n$-dimensional compact Calabi-Yau
manifold and $\{\varphi_{1}, \cdots,
\varphi_{N}\}\in\H^{0,1}(X,T^{1,0}_X)$ a basis  where  $N=\dim
\H^{0,1}(X,T^{1,0}_X)$.  As constructed in Theorem \ref{mPhi}, there
exists a smooth family of Beltrami differentials in the following
form
\[
\Phi(t)=\sum_{i=1}^{N}\varphi_{i}t_{i}+\sum_{|I|\geq2}\varphi_{I}t^{I}
=\sum_{\nu_{1}+\cdots+\nu_{N}\geq1}\varphi_{\nu_{1}\cdots\nu_{N}}t^{\nu_{1}%
}_{1}\cdots t^{\nu_{N}}_{N}\in A^{0,1}(X,T^{1,0}_X)
\]
for $t\in \mathbb C^N$ with $|t|<1$.
 It is easy to check that the map
\beq e^{\Phi(t)}\lrcorner:\, A^{0}(X,K_{X})\rightarrow
A^{0}(X_{t},K_{X_{t}})\label{iso}\eeq is a well-defined linear
isomorphism.

\begin{proposition} \label{holcriteria}For any smooth $(n,0)$ form
$\Omega\in A^{n,0}(X)$, the section $e^{\Phi(t)}\lrcorner\Omega\in
A^{n,0}(X_{t})$ is holomorphic with respect to the complex structure
$J_{\Phi(t)}$ induced by $\Phi(t)$ on $X_{t}$ if and only if
\begin{equation}
\label{m=1}\overline{\partial}\Omega+\partial(\Phi(t)\lrcorner\Omega)=0.
\end{equation}
\begin{proof}
This is a direct consequence of Corollary \ref{key2}. In fact,
$$\left(e^{- i_\Phi}\circ d \circ e^{
i_\Phi}\right)(\Omega)=\bp\Omega+\p(\Phi\lrcorner
\Omega)\label{rec1},$$ if the vector bundle $E$ is trivial
 and $\Phi(t)$ satisfies the integrability condition. The
operator $d$, which is independent of the complex structures, can be
decomposed as $d=\overline{\partial}_t+{\partial}_t$, where
$\overline{\partial}_t$ and ${\partial}_t$ denote the $(0,1)$-part
and $(1,0)$-part of $d$, with respect to the complex structure
$J_{\Phi(t)}$ induced by $\Phi(t)$ on $X_{t}$. Note that
$e^{\Phi(t)}\lrcorner\Omega\in A^{n,0}(X_{t})$ and so $$ \p_t
\left(e^{i_{\Phi}}(\Omega)\right)=\p_t (e^{\Phi(t)}\lrcorner
\Om)=0.$$ Hence,
$$\left(e^{-
i_\Phi}\circ \overline{\partial}_t \circ e^{
i_\Phi}\right)(\Omega)=\bp\Omega+\p(\Phi\lrcorner
\Omega)\label{rec1},$$ which implies the assertion. (In case
$\Phi(t)$ is just $L^2$-integrable, we also see from this formula
that $\bp_t\left(e^{\Phi(t)}\lrcorner\Om\right)=0$ in the
distribution sense if $\bp\Omega+\p(\Phi\lrcorner \Omega)=0$ in the
distribution sense, and so by hypoellipticity of $\bp_t$ on $(n,0) $
forms of $X_t$, we know $e^{\Phi(t)}\lrcorner \Omega$ is, in fact, a
holomorphic $(n,0)$-form on $X_t$.)
\end{proof}

\end{proposition}

\begin{theorem}
\label{gcfn} Let $\Omega_0$ be a nontrivial holomorphic $(n,0)$-form
on the CY manifold $X$ and  $X_{t}=(X_{t},J_{\Phi(t)})$ be the
deformation of the CY manifold $X$ induced by the $L^{2}$-global
canonical family $\Phi(t)$ of Beltrami differentials on $X$ as
constructed in Theorem \ref{mPhi}. Then, for
$|t|<1$, \beq \Omega_{t}^{C}:=e^{\Phi(t)}\lrcorner\Omega_{0}\eeq defines an $L^{2}%
$-global canonical family of holomorphic $(n,0)$-forms on $X_{t}$
and depends on $t$ holomorphically.
\end{theorem}

\begin{proof} Since $\Omega_0$ is holomorphic, and $\Phi(t)$ is
smooth, by (\ref{hol}), we obtain
$$\overline{\partial}\Omega_0+\partial(\Phi(t)\lrcorner\Omega_0)=0.$$
Hence, by Proposition \ref{holcriteria} and Theorem \ref{mPhi},
$\Omega_{t}^{C}=e^{\Phi(t)}\lrcorner\Omega_{0}$ defines an $L^{2}%
$-global canonical family of holomorphic $(n,0)$-forms on $X_{t}$
for $|t|<1$.
The holomorphic dependence of $\Phi(t)$ on $t$ implies that $\Omega_{t}%
^{C}$ depends on $t$ holomorphically.
\end{proof}

\bcorollary\label{edeRK} Let
$\Omega_{t}^{C}:=e^{\Phi(t)}\lrcorner\Omega_{0}$ be the
$L^{2}$-global canonical family of holomorphic $(n,0)$-forms as
constructed in Theorem \ref{gcfn}.  Then for $|t|<1$, there holds
the following global expansion of $[\Omega_{t}^{C}]$ in cohomology
classes,
$$
[\Omega_{t}^{C}]=[\Omega_{0}]+\sum_{i=1}^{N}[\varphi_{i}\lrcorner\Omega
_{0}]t_{i}+O(|t|^{2}).
$$
where $O(|t|^{2})$ denotes the terms in
$\displaystyle\bigoplus_{j=2}^{n}H^{n-j,j}(X)$ of orders at least
$2$ in $t$.
\begin{proof}
From Theorem \ref{gcfn} and Hodge theory we can see that for
$|t|<1$,
\begin{align*}
[\Omega_{t}^{C}]  &
=[\Omega_{0}]+\sum_{i=1}^{N}[\mathbb{H}(\varphi_{i}\lrcorner
\Omega_{0})]t_{i}+\sum_{|I|\geq2}
[\mathbb{H}(\varphi_{I}\lrcorner\Omega
_{0})]t^{I}+\sum_{k\geq2}\frac{1}{k!}\left[\mathbb{H}\Big(\bigwedge^{k}%
\Phi(t)\lrcorner\Omega_{0}\Big)\right]
\end{align*}
By Theorem \ref{mPhi}, $\varphi_{i}\lrcorner\Omega_{0}$ is harmonic
and that $\varphi_{I}\lrcorner\Omega_{0}$ is $\partial$-exact for
each $|I|\geq2$. Hence
$$[\Omega_{t}^{C}]
=[\Omega_{0}]+\sum_{i=1}^{N}[\varphi_{i}\lrcorner \Omega_{0}]t_{i}+
O(|t|^2)$$ where $O(|t|^2)$ denotes the term $\displaystyle \sum_{k\geq2}\frac{1}{k!}\left[\mathbb{H}\Big(\bigwedge^{k}%
\Phi(t)\lrcorner\Omega_{0}\Big)\right]\in\displaystyle\bigoplus_{j=2}^{n}
H^{n-j,j}(X)$.
\end{proof}

\ecorollary

\subsection{Iteration procedure on deformation spaces of compact K\"ahler manifolds}
\label{iteration}
 In this
subsection, we extend our constructions to the deformation spaces of
compact K\"ahler manifolds. We shall use iteration procedure to
construct holomorphic sections of the canonical line bundle
$K_{X_{t}}$ of the deformation $X_{t}$ of a K\"{a}hler manifold $X$
induced by the Beltrami differential $\Phi(t)$ satisfying the
integrability condition. More precisely, our goal is to find a
convergent power series for any holomorphic section $\Omega_{0}\in
H^{0}(X,K_{X})$,
$$
\Omega_{t}=\Omega_{0}+\sum_{|I|\geq1} t^{I}\Omega_{I}%
$$
such that $e^{\Phi(t)}\lrcorner\Omega_{t}\in H^{0}(X_{t},K_{X_{t}})$
is holomorphic with respect to the induced complex structure
$J_{\Phi(t)}$ by $\Phi(t)$.

Let $X$ be an $n$-dimensional compact K\"ahler manifold and
$\{\varphi_{1}, \cdots, \varphi_{N}\}\in\H^{1}(X,T^{1,0}_X)$ a basis
with the norm $\|\varphi_{i}\|=C_N$, for each $i=1, 2,\cdots $ where
$N=\dim \H^{1}(X,T^{1,0}_X)$. In general, on deformation spaces of
compact K\"ahler manifolds, we can not construct  Beltrami
differentials $\Phi(t)$ as stated in Theorem \ref{Phi} or Theorem
\ref{mPhi}, where we essentially use the non-where vanishing
property of $\Om_0$ on Calabi-Yau manifolds. Hence, it is natural to
make the following definition.

\bdefinition A power series of Beltrami differentials of the
following form
\[
\Phi(t)=\sum_{i=1}^{N}\varphi_{i}t_{i}+\sum_{|I|\geq2}\varphi_{I}t^{I}
=\sum_{\nu_{1}+\cdots+\nu_{N}\geq1}\varphi_{\nu_{1}\cdots\nu_{N}}t^{\nu_{1}%
}_{1}\cdots t^{\nu_{N}}_{N}\in A^{0,1}(X,T^{1,0}_X)
\]
with $\varphi_{0\cdots\nu_{i}\cdots0}=\varphi_{i}$, is called an {$L^{2}%
$-global canonical family of Beltrami differentials on the K\"ahler
manifold} $X$ if it satisfies:

\bd \item the integrability condition: $\overline{\partial}\Phi(t)=\frac{1}{2}%
[\Phi(t),\Phi(t)]$;

\item  global convergence in the sense that
\[
\|\Phi(t)\lrcorner\Omega_{0}\|_{L^{2}}\leq\sum_{|I|\geq1}\|\varphi
_{I}\|\|\Omega_{0}\|\cdot t^{|I|}<\infty
\]
as long as $t\in \mathbb C^N$ with $|t|< R$, where the convergence
radius $R$ is a constant only depending on $C_N$ and $\Omega_{0}$ is
a non-zero holomorphic $(n,0)$-form. \ed \edefinition

As an analogue to Theorem \ref{gcfn} on deformation spaces of CY
manifolds, we have the following result on deformation spaces of
compact K\"ahler manifolds:
\begin{theorem}
\label{m1} If there exists an $L^{2}$-global canonical family
$\Phi(t)$ of Beltrami differentials on the K\"ahler manifold $X$
with convergence radius $R$, and let $X_{t}=(X_{t}, J_{\Phi(t)})$ be
the deformation of $X$ induced by $\Phi(t)$,  then for any
holomorphic $(n,0)$-form $\Omega$, we can construct a smooth power
series
\begin{equation}
\label{ps}\Omega_{t}=\Omega_{0}+\sum_{|I|\geq1}^{\infty}\Omega_{I}t^{I}
\in A^{n,0}(X)
\end{equation}
such that $\Omega_{0}=\Omega$ with the following properties:

$a)$ $\Omega_{t}^{C}:=e^{\Phi(t)}\lrcorner\Omega_{t}\in H^{0}(X_{t},K_{X_{t}%
})$ is holomorphic with respect to $J_{\Phi(t)}$;

$b)$ $\Omega_{I}\in A^{n,0}(X)$ is $\partial$-exact and also
$\overline {\partial}^{*}$-exact for all $|I|\geq1$;

\end{theorem}

\begin{proof}
[Proof] By the proof of Proposition \ref{holcriteria}, we see it
also holds on compact K\"ahler manifold $X$. Hence by Proposition
\ref{holcriteria}, we know that $\Omega_{t}$ must satisfy the
equation
\begin{equation}
\label{holcond}\overline{\partial}\Omega_{t}=-\partial(\Phi(t)\lrcorner
\Omega_{t}).
\end{equation}
By comparing the coefficients of $t_{1}^{\nu_{1}}\cdots
t_{N}^{\nu_{N}}$ of both sides of (\ref{holcond}), one knows that
Equation (\ref{holcond}) is equivalent to
\begin{equation}
\label{holt}%
\begin{cases}
\overline{\partial}\Omega_{0}=0,\\
\overline{\partial}\Omega_{\nu_{1}\cdots\nu_{N}}=-\partial\left(
\sum\limits_{\substack{\alpha_{i}+\beta_{i}=\nu_{i},\alpha_{i}\geq0\\}%
}\varphi_{\alpha_{1}\cdots\alpha_{N}}\lrcorner\Omega_{\beta_{1}\cdots\beta
_{N}}\right)  ,
\end{cases}
\end{equation}
where each $\nu_{i}\geq0$ and $\Sigma\nu_{i}\geq1$.

We first prove that the equation (\ref{holt}) has a $\partial$-exact
solution by induction. Set
\[
\eta_{\nu_{1}\cdots\nu_{N}}=-\partial\left(
\sum\limits_{\substack{\alpha
_{i}+\beta_{i}=\nu_{i},\alpha_{i}\geq0}}\varphi_{\alpha_{1}\cdots\alpha_{N}%
}\lrcorner\Omega_{\beta_{1}\cdots\beta_{N}}\right)  ,
\]
which is clearly $\partial$-exact and thus $\mathbb{H}_{\overline{\partial}%
}(\eta)=0$ by the K\"ahler identity
$\square_{\partial}=\square_{\overline {\partial}}$. So by
$\overline{\partial}$-Inverse Lemma \ref{eq} it suffices to show
that $\overline{\partial}\eta_{\nu_{1}\cdots\nu_{N}}=0$.

For the initial case $\Sigma\nu_{i}=1$, one has
\[
\overline{\partial}\eta_{\nu_{1}\cdots\nu_{N}}=-\overline{\partial}%
\partial(\varphi_{\nu_{1}\cdots\nu_{N}}\lrcorner\Omega_{0}) =\partial
(\overline{\partial}\varphi_{\nu_{1}\cdots\nu_{N}}\lrcorner\Omega_{0}%
+\varphi_{\nu_{1}\cdots\nu_{N}}\lrcorner\overline{\partial}\Omega_{0})=0
\]
since $\overline{\partial}\varphi_{\nu_{1}\cdots\nu_{N}}=0$ and
$\overline {\partial}\Omega_{0}=0$. Thus we have
\begin{equation}
\label{cfn01}\Omega_{\nu_{1}\cdots\nu_{N}}=\overline{\partial}^{*}%
\mathbb{G}\eta_{\nu_{1}\cdots\nu_{N}}=-\overline{\partial}^{*}\partial
\mathbb{G}(\varphi_{\nu_{1}
\cdots\nu_{N}}\lrcorner\Omega_{0})=\partial
\overline{\partial}^{*} \mathbb{G}(\varphi_{\nu_{1}\cdots\nu_{N}}%
\lrcorner\Omega_{0})
\end{equation}
by $\overline{\partial}$-Inverse Lemma \ref{eq} and K\"{a}hler
identity.

Supposing that the $(n,0)$-forms $\Omega_{\nu_{1}\cdots\nu_{N}}$
with $\Sigma\nu_{i}=K$ are constructed, we can also prove
$$\overline{\partial} \eta_{\nu_{1}\cdots\nu_{N}}=0$$ for
$\Sigma\nu_{i}=K+1$ by induction and the commutator formula Lemma
\ref{TT3}. This calculation is routine and left to the interested
readers. Similar to the initial case, we can construct the
$(n,0)$-forms $\Omega_{\nu_{1}\cdots\nu_{N}}$ with
$\Sigma\nu_{i}=K+1$ as
\[
\Omega_{\nu_{1}\cdots\nu_{N}} =-\overline{\partial}^{*}\partial\mathbb{G}%
\left(
\sum\limits_{\substack{\alpha_{i}+\beta_{i}=\nu_{i},\alpha_{i}\geq
0}}\varphi_{\alpha_{1}\cdots\alpha_{N}}\lrcorner\Omega_{\beta_{1}\cdots
\beta_{N}}\right)  =\partial\overline{\partial}^{*} \mathbb{G}\left(
\sum\limits_{\substack{\alpha_{i}+\beta_{i}=\nu_{i},\alpha_{i}\geq0}%
}\varphi_{\alpha_{1}\cdots\alpha_{N}}\lrcorner\Omega_{\beta_{1}\cdots\beta
_{N}}\right)  .
\]
Hence we have completed the construction of the power series
$\Omega_{t}$ of $(n,0)$-forms.

Finally, let us prove the global convergence of the formal power
series. By the global convergence of the canonical family of
Beltrami differentials, we know that there exists a small constant
$\xi>0$ and a constant $R_{1}\in(0,R]$ such that
\[
\sum_{|I|=i}\|\varphi_{I}\|R_{1}^{i}\leq\xi
\]
for all large $i>0$. We may assume that this fact holds for all
$i>0$. Then we have the following estimate for each $i>0$
\begin{equation}
\label{OI}\sum_{|I|=i}\|\Omega_{I}\|\leq\xi(\xi+1)^{i-1}R_{1}^{-i},
\end{equation}
which follows by induction and implies the convergence of power
series (\ref{ps}) as long as $|t|< R_{1}$. We set $\|\Omega_{0}\|=1$
for convenience. First for the initial case $i=1$, one has
\[
\sum_{|I|=1}\|\Omega_{I}\|\leq\|\Omega_{0}\|\sum_{|I|=1}\|\varphi_{I}\|\leq
R_{1}^{-1}\xi,
\]
where the quasi-isometry Theorem \ref{1main} is applied. Then, we
assume that the estimate (\ref{OI}) is true for $l=1, \cdots, i-1$
and try to prove the case $l=i$ as follows.
\begin{align*}
\sum_{|I|=i}\|\Omega_{I}\|  &  \leq\sum\limits_{\substack{|I|=i,|I_{2}%
|\geq1,\\I_{1}+I_{2}=I}}\|\Omega_{I_{1}}\|\cdot\|\varphi_{I_{2}}\|\\
&  \leq\xi R_{1}^{-1}\xi(\xi+1)^{i-2}R_{1}^{-(i-1)}+ \cdots+\xi R_{1}^{-i}%
\xi+\xi R_{1}^{-i}\\
&  =(\xi R_{1}^{-i})\xi\frac{1-(\xi+1)^{i-1}}{1-(\xi+1)} +\xi R_{1}^{-i}\\
&  = \xi(\xi+1)^{i-1}R_{1}^{-i},
\end{align*}
where the first inequality is also due to Theorem \ref{1main}. Yet
it is easy to check that the convergence domain for $|t|$ of
$\sum_{i=1}\xi (\xi+1)^{i-1}R_{1}^{-i}|t|^{i}$ is obviously
$[0,R_{1})$.

The regularity of $\Omega_{t}$ follows  by similar arguments as in
the proof of Theorem \ref{Phi}. This completes the proof of Theorem
\ref{m1}.
\end{proof}

As similar as Corollary \ref{edeRK}, we  also obtain a global
expansion of the canonical family of $(n,0)$-forms on the
deformation spaces of compact K\"ahler manifolds in cohomology
classes.
\begin{corollary}
\label{edeRK1} Let $\Omega_{t}^{C}:=e^{\Phi(t)}\lrcorner\Omega_{t}$
be the $L^{2}$-global canonical family of holomorphic $(n,0)$-forms
as constructed in Theorem \ref{m1}. Then for $|t|<R$, there holds
the following global expansion of the de Rham cohomology classes of
it

$$
\lbrack\Omega_{t}^{C}]=[\Omega_{0}]+\sum_{|I|\geq 1}
[\mathbb{H}(\varphi_{I}\lrcorner\Omega _{0})]t^{I}+O(|t|^{2}),
$$
where $O(|t|^{2})$ denotes the terms in
$\displaystyle\bigoplus_{j=2}^{n}H^{n-j,j}(X)$ of orders at least
$2$ in $t$.
\end{corollary}
\begin{proof}
The proof  is very similar to that of Corollary \ref{edeRK}.
\begin{align*}
[\Omega_{t}^{C}]  &
=[\Omega_{0}]+\sum_{i=1}^{N}[\mathbb{H}(\varphi_{i}\lrcorner
\Omega_{0})]t_{i}+\sum_{|I|\geq2}
[\mathbb{H}(\varphi_{I}\lrcorner\Omega
_{0})]t^{I}+\sum_{k\geq2}\frac{1}{k!}\left[\mathbb{H}\Big(\bigwedge^{k}%
\Phi(t)\lrcorner\Omega_{0}\Big)\right]
\end{align*}
The difference is that,  $\varphi_{i}\lrcorner\Omega_{0}$ is not
necessarily harmonic,  and  for  $|I|\geq2$
$\varphi_{I}\lrcorner\Omega_{0}$ is
 not $\partial$-exact in general.
\end{proof}

\end{document}